\pdfoutput=1
\newif\ifpersonal
\RequirePackage[l2tabu, orthodox]{nag} 
\documentclass[12pt,a4paper,reqno]{amsart} 
\usepackage{amsmath,amsthm,amssymb,mathrsfs,mathtools,bm,eucal,tensor} 
\usepackage{enumerate,comment,braket,xspace,tikz-cd,csquotes} 
\usepackage[utf8]{inputenc} 
\usepackage[centering,vscale=0.7,hscale=0.7]{geometry}
\usepackage[hidelinks]{hyperref}
\usepackage[capitalize]{cleveref}

\ifpersonal
\newcommand*{\personal}[1]{\textcolor{blue}{(Personal: #1)}}
\newcommand*{\todo}[1]{\textcolor{red}{(Todo: #1)}}
\newcommand*{\question}[1]{\textcolor{purple}{(Question: #1)}}
\else
\newcommand*{\personal}[1]{}
\newcommand*{\todo}[1]{}
\newcommand*{\question}[1]{}
\fi

\theoremstyle{plain}
\newtheorem{thm-intro}{Theorem}
\newtheorem{thm}{Theorem}[section]
\newtheorem{lem}[thm]{Lemma}
\newtheorem{prop}[thm]{Proposition}

\theoremstyle{definition}
\newtheorem{defin}[thm]{Definition}

\newtheorem{eg}[thm]{Example}
\newtheorem{rem}[thm]{Remark}
\theoremstyle{remark}
\numberwithin{equation}{section}

\newcommand{\rL}{\mathrm L}

\newcommand{\cC}{\mathcal C}

\newcommand{\cO}{\mathcal O}
\newcommand{\cT}{\mathcal T}
\newcommand{\cS}{\mathcal S}
\newcommand{\cX}{\mathcal X}
\newcommand{\cY}{\mathcal Y}

\newcommand{\cG}{\mathcal G}

\newcommand{\cD}{\mathcal D}

\newcommand{\bP}{\mathbf P}

\newcommand{\PSh}{\mathrm{PSh}}
\newcommand{\Sh}{\mathrm{Sh}}

\newcommand{\Top}{\mathcal T\mathrm{op}}
\newcommand{\inv}{^{-1}}

\newcommand{\canal}{$\mathbb C$-analytic\xspace}

\DeclareMathAlphabet{\mathpzc}{OT1}{pzc}{m}{it}

\newcommand{\sSet}{\mathrm{sSet}}
\newcommand{\rSet}{\mathrm{Set}}

\newcommand{\tauet}{\tau_\mathrm{\acute{e}t}}

\newcommand{\cTdisck}{\cT_{\mathrm{disc}}(k)}
\newcommand{\cTet}{\cT_{\mathrm{\acute{e}t}}}
\newcommand{\cTetk}{\cT_{\mathrm{\acute{e}t}}(k)}
\newcommand{\cGetk}{\cG_{\mathrm{\acute{e}t}}(k)}
\newcommand{\cGetkder}{\cG_{\mathrm{\acute{e}t}}^\mathrm{der}(k)}

\newcommand{\RTop}{\mathcal{RT} \mathrm{op}}

\newcommand{\Str}{\mathrm{Str}}

\newcommand{\CRing}{\mathrm{CRing}}
\newcommand{\DM}{Deligne-Mumford\xspace}
\newcommand{\sMod}{\mathrm{sMod}}
\newcommand{\sAlg}{\mathrm{sAlg}}
\newcommand{\dAff}{\mathrm{dAff}}
\newcommand{\et}{\mathrm{\acute{e}t}}
\newcommand{\Sch}{\mathrm{Sch}}
\newcommand{\SpecEt}{\Spec^{\et}}

\DeclareMathOperator{\Hom}{Hom}

\DeclareMathOperator{\Map}{Map}
\DeclareMathOperator{\Fun}{Fun}

\DeclareMathOperator{\Spec}{Spec}

\DeclareMathOperator*{\colim}{colim}

\begin{document}

\title{Comparison results for derived Deligne-Mumford stacks}

\author{Mauro PORTA}
\address{Mauro PORTA, Institut de Math\'ematiques de Jussieu, CNRS-UMR 7586, Case 7012, Universit\'e Paris Diderot - Paris 7, B\^atiment Sophie Germain 75205 Paris Cedex 13 France}
\email{mauro.porta@imj-prg.fr}
\date{July 1, 2015 (Revised on \today)}


\begin{abstract}
	In this short note we write down a comparison between the notion of derived Deligne-Mumford stack in the sense of \cite{HAG-II} and the one introduced in \cite{DAG-V}.
	It is folklore that the two theories yield essentially the same objects, but it is difficult to locate in the literature a precise result, despite being sometimes useful to be able to switch between the two frameworks.
\end{abstract}

\maketitle

\personal{PERSONAL COMMENTS ARE SHOWN!!!}

\tableofcontents

\section*{Introduction}

This short note is devoted to establishing in a precise way the folklore equivalence between the theory of derived \DM stacks introduced by B.\ To\"en and G.\ Vezzosi in \cite{HAG-II} and the one defined by J.\ Lurie in \cite{DAG-V}. The main comparison result will be stated in the next section, see \cref{thm:comparison}.
I claim essentially no originality for the results contained in this paper, except perhaps for the exposition.
Indeed, even though \cref{thm:comparison} had appeared nowhere in the literature (at least to the best of our knowledge), there are many hints scattered through the DAG series of J.\ Lurie that leave absolutely no doubt about his knowledge of the precise terms of the comparison.
We will occasionally redirect the reader there.

This note will be hardly of any importance for the community, except perhaps for sake of a written reference.
However, it could still be helpful for someone who is trying to approach the subject of derived algebraic geometry for the first time.
For this reason, I preferred to be lengthy and to give careful explanations even where perhaps they wouldn't have been necessary.

\subsection*{Conventions}

Throughout this note we will work freely with the language of $(\infty,1)$-categories.
We will call them simply $\infty$-categories and our basic reference on the subject is \cite{HTT}.
Occasionally, it will be necessary to consider $(n,1)$-categories. We will refer to such objects as $n$-categories, and we redirect the reader to \cite[§2.3.4]{HTT} for the definitions and the basic properties. There won't be any chance of confusion with the theory of $(\infty,n)$-categories, since it plays no role in this note.
The notation $\cS$ will be reserved for the $\infty$-categories of spaces.
Whenever categorical constructions are used (such as limits, colimits etc.), we mean the corresponding $\infty$-categorical notion.
For the reader with a model categorical background, this means that we are always considering \emph{homotopy} limits, \emph{homotopy} colimits etc.
See \cite[4.2.4.1]{HTT}.

In \cite{HTT} and more generally in the DAG series, whenever $\cC$ is a $1$-category the notation $\mathrm N(\cC)$ denotes $\cC$ reviewed (trivially) as an $\infty$-category. This notation stands for the nerve of the category $\cC$ (and this is because an $\infty$-category in \cite{HTT} is defined to be a quasicategory, that is a simplicial set with special lifting properties). In this note, we will systematically suppress this notation, and we encourage the reader to think to $\infty$-categories as model-independently as possible. For this reason, if $k$ is a (discrete) commutative ring we chose to denote by $\mathrm{CRing}_k$ the $1$-category of discrete $k$-algebras and by $\mathrm{CAlg}_k$ the $\infty$-category underlying the category of simplicial commutative $k$-algebras.

\subsection*{Acknowledgments}

I felt the need to look for a precise comparison result after a discussion I had with Marco Robalo.
I take the opportunity to thank him for all the interesting conversations we had during this year.
I would also like to thank my advisor Gabriele Vezzosi for introducing me to such an interesting topic as derived geometry, under all its facets.

\section{Statement of the comparison result}

Let us start by quickly reviewing the two theories.

\subsection{HAG II framework} \label{subsec:HAG_II}

In \cite{HAG-II} the authors work within the setting previously introduced in \cite{HAG-I}, where the theory of model topoi is introduced and extensively explored. 
This means that model categories are used continuously throughout the whole paper.
In order to compare their constructions with the ones of \cite{DAG-V} it will be convenient to rethink the paper in a purely $\infty$-categorical language.
This is essentially no more than an easy exercise, and we take the opportunity of this review to explain how it can be done.

Let $k$ be a commutative ring (with unit). We will denote by $\sMod_k$ the category of simplicial $k$-modules.
There is an adjunction
\[ U \colon \sMod_k \rightleftarrows \sSet \colon F \qquad (F \dashv U) \]
which satisfies the hypothesis of the lifting principle (see \cite{Schwede_Shipley_Algebras_1998}) and therefore it allows to lift the (Kan) model structure on $\sSet$ to a simplicial model structure on $\sMod_k$. Moreover, with respect to this model structure, $\sMod_k$ becomes a monoidal model category (whose tensor product is computed objectwise).
We set $\sAlg_k \coloneqq \mathrm{Com}(\sMod_k)$.
Using the fact that every object in $\sMod_k$ is fibrant, it is possible to establish that the adjunction
\[ V \colon \sAlg_k \rightleftarrows \sMod_k \colon \mathrm{Sym}_k \qquad (\mathrm{Sym}_k \dashv V) \]
satisfies again the lifting principle (see \cite[§5]{Schwede_Shipley_Algebras_1998}), and therefore the (simplicial) model structure on $\sMod_k$ induces a simplicial model structure on $\sAlg_k$.
We will simply denote by $\mathrm{CAlg}_k$ the $\infty$-category underlying $\sAlg_k$, which can be explicitly thought as the coherent nerve \cite[§1.1.5]{HTT} of the category of fibrant cofibrant objects in $\sAlg_k$. It is customary to denote the opposite of this $\infty$-category by $\mathrm{dAff}_k$ (the $\infty$-category of ``affine derived schemes'').

This $\infty$-category admits another description which is more useful for our purposes. Let $\cTdisck$ the full subcategory of ordinary schemes over $\Spec(k)$ spanned by the finite dimensional relative affine spaces $\mathbb A^n_k$. We can think of $\cTdisck$ as a (one-sorted) Lawvere theory; or, with the language of \cite{DAG-V}, we can equally say that $\cTdisck$ is a \emph{discrete pregeometry}. The $\infty$-category of product preserving functors with values in the $\infty$-category of spaces can be identified with the \emph{sifted completion} of $\cTdisck$ and we will denote it by $\mathcal P_\Sigma(\cTdisck)$ (see \cite[Definition 5.5.8.8]{HTT}). This is a presentable $\infty$-category and therefore it admits a presentation by a model category \cite[A.3.7.6]{HTT}, which can be easily obtained as follows: consider the category of simplicial presheaves on $\cTdisck$ endowed with the global projective model structure.
Then the underlying $\infty$-category of the Bousfield localization of this  model category at the collection of maps $y(\mathbb A^{n}_k) \coprod y(\mathbb A^m_k) \to y(\mathbb A^{n+m}_k)$ (where $y$ denotes the Yoneda embedding) precisely coincides with $\mathcal P_{\Sigma}(\cTdisck)$.
It is somehow remarkable that $\mathcal P_\Sigma(\cTdisck)$ admits a much stricter presentation.
Consider in fact the category of functors $\cTdisck \to \sSet$ which \emph{strictly} preserve products. It follows from a theorem of Quillen \cite[5.5.9.1]{HTT} that this simplicial category admits a global projective model structure. Moreover, a theorem of J.\ Bergner \cite[5.5.9.2]{HTT} shows that the underlying $\infty$-category coincides precisely with $\mathcal P_\Sigma(\cTdisck)$.
However, the category of product preserving functors $\cTdisck \to \sSet$ is precisely equivalent to $\sAlg_k$, and the two model structures agree.
Therefore, we have a categorical equivalence
\[ \mathrm{CAlg}_k \simeq \mathcal P_\Sigma(\cTdisck) \]
The reader might want to consult also \cite[Remark 4.1.2]{DAG-V} for another discussion of this equivalence.

The next step is to introduce the \'etale topology on the model category $\sAlg_k$. As this notion only depends on the homotopy category of $\sAlg_k$ (cf.\ \cite[Definition 4.3.1]{HAG-I}) it also defines a Grothendieck topology on the $\infty$-category $\mathrm{CAlg}_k$ (cf.\ \cite[6.2.2.3]{HTT}).
We briefly recall that a morphism $f \colon A \to B$ in $\sAlg_k$ is said to be \'etale if $\pi_0(f) \colon \pi_0(A) \to \pi_0(B)$ is \'etale and the canonical map
\[ \pi_i(A) \otimes_{\pi_0(A)} \pi_0(B) \to \pi_i(B) \]
is an isomorphism (that is, the morphism is strong).
Similarly, a morphism $f \colon A \to B$ is smooth if it is strong and $\pi_0(f) \colon \pi_0(A) \to \pi_0(B)$ is smooth.
We will denote by $\tauet$ the \'etale topology and by $\bP_\et$ (resp.\ $\bP_{\mathrm{sm}}$) the collection of \'etale (resp.\ smooth) morphisms.
Using these data, one can form the model category of hypersheaves with respect to the \'etale topology. Recall that this is obtained in the following two steps:
\begin{enumerate}
	\item consider the global projective model structure on $\mathrm{Funct}(\sAlg_k, \sSet)$;
	\item consider next the Bousfield localization of this model structure at the collection of hypercovers (see \cite[§4.4 and §4.5]{HAG-I} or \cite[§6.5.3]{HTT}).
\end{enumerate}
The result is what is denoted in \cite{HAG-II} by $\dAff^{\sim, \tauet}$. It follows from \cite[6.5.2.14, 6.5.2.15]{HTT} that the underlying $\infty$-category of $\dAff^{\sim, \tauet}$ can be simply identified with the hypercompletion $\Sh(\dAff_k, \tauet)^\wedge$ (we refer the reader to \cite[§6.5.2]{HTT} for a detailed discussion of this notion). We will usually call the objects in $\Sh(\dAff_k, \tauet)^\wedge$ as stacks (for the \'etale topology).
The next step is to consider geometric stacks inside $\Sh(\dAff_k, \tauet)^\wedge$.
Since there are many references for this subject \cite{Simpson_Algebraic_1996,HAG-II,Toen_Algebrisation_2008,Porta_Yu_Higher_analytic_stacks_2014}, we won't repeat the full definition here, but we will limit ourselves to
Roughly speaking, geometric stacks are stacks $X$ admitting an atlas $p \colon U \to X$, that is an effective epimorphism $p$ (see \cite[§6.2.3]{HTT} and the very useful \cite[7.2.1.14]{HTT}) whose source $U$ is an affine derived scheme (seen as a stack via the $\infty$-categorical Yoneda embedding, see \cite[§5.1.3]{HTT} or \cite[§5.2.1]{Lurie_Higher_algebra}) and which is in $\bP_\et$ or in $\bP_{\mathrm{sm}}$.
In the first case, we will refer to the stack as a (higher) derived \DM stack, and in the latter as a (higher) derived Artin stack.
In this note, we will be only concerned with derived \DM stacks. We will denote the full subcategory of $\Sh(\dAff_k, \tauet)^\wedge$ spanned by derived \DM stacks by $\mathbf{DM}$.
Let us complete the review of \cite{HAG-II} with the following two additional remarks:

\begin{enumerate}
	\item Geometric stack is always stable under weak equivalences because only homotopy-invariant categorical constructs are used in formulating it (i.e.\ homotopy coproducts, homotopy geometric realizations etc.). Therefore \cite[4.2.4.1]{HTT} shows that the notion of geometric stack can be equally formulated at the level of the $\infty$-category $\Sh(\dAff_k, \tauet)^\wedge$.
	\item The category $\mathbf{DM}$ is naturally filtered by the notion of geometric level: a stack is said to be $(-1)$-geometric if it is representable by an object in $\dAff_k$. If $A \in \mathrm{CAlg}_k$, we choose to represent its functor of points by $\Spec(A) \in \mathbf{DM} \subset \Sh(\dAff_k, \tauet)^\wedge$.
	Next, proceeding by induction, we will say that a stack $X$ is $n$-geometric if it admits an atlas $p \colon U \to X$ which is representable by $(n-1)$-geometric stacks in the following precise sense: for every representable stack $\Spec(A)$ and any map $\Spec(A) \to X$ the base change $\Spec(A) \times_X U$ is $(n-1)$-geometric. We will denote by $\mathbf{DM}_n$ the full subcategory of $\mathbf{DM}$ spanned by $n$-geometric derived \DM stacks whose restriction to $\CRing_k$ is an $n$-truncated stack (i.e.\ it takes values in $n$-truncated spaces).
\end{enumerate}

\subsection{DAG V framework} \label{subsec:DAG_V}

The point of view taken in \cite{DAG-V} is quite different.
We refer the reader to the introduction of \cite{Porta_GAGA_2015} for an expository account of the role of (pre)geometries (cf.\ \cite[§1.2, 3.1]{DAG-V}) in the construction of affine derived objects.
Here, we will content ourselves with a short review of the theory of $\cG$-schemes for a given geometry $\cG$ from the point of view of \cite{DAG-V}.
Recall either from \cite[Definition 12.8]{DAG-V} or from the introduction of \cite{Porta_GAGA_2015} that a geometry is an $\infty$-category $\cG$ with finite limits and equipped with some extra structure, consisting of a collection of ``admissible'' morphisms and a Grothendieck topology $\tau$ on $\cG$ generated by admissible morphisms.
If $\cX$ is an $\infty$-topos and $\cG$ is a geometry, it is defined an $\infty$-category of $\cG$-structures, denoted $\Str_\cG(\cX)$.
Recall that a $\cG$-structure is a functor $\cG \to \cX$ which is left exact and takes $\tau$-coverings to effective epimorphisms in $\cX$.

Before moving on, it is important to discuss a very important special case.
If $\cX$ is the $\infty$-topos of $\cS$-valued sheaves on some topological space $X$, we can think of a $\cG$-structure on $\cX$ as a sheaf on $X$ with values in the $\infty$-category $\mathrm{Ind}(\cG^{\mathrm{op}})$ having special behavior on the stalks, as the next key example shows:

\begin{eg} \label{eg:discrete_etale_geometry}
	Let $k$ be a fixed (discrete) commutative ring. We denote by $\cGetk$ to be the category $(\mathrm{CRing}_k^{\mathrm{f.p.}})^{\mathrm{op}}$, the opposite of the category of discrete $k$-algebras of finite presentation.
	Moreover, we declare a morphism in $\cGetk$ to be an admissible morphism if and only if it is \'etale, and we endow $\cGetk$ with the usual \'etale topology.
	In this case, $\mathrm{Ind}(\cGetk^{\mathrm{op}}) \simeq \mathrm{CRing}_k$, the category of discrete $k$-algebras of finite presentation.
	Then a $\cGetk$-structure $\cO$ on $\Sh(X)$ is a sheaf of discrete commutative rings on $X$ whose stalks are strictly henselian local rings.
	The fact that $\cO$ has to be discrete follows from his left exactness (see \cite[§5.5.6]{HTT} for a general discussion of truncated objects in an $\infty$-category and more specifically \cite[5.5.6.16]{HTT} for the needed property).
	The statement on stalks, instead, is due to the following fact: for every point $x \in X$ (formally seen as a geometric morphism $x\inv \colon \Sh(X) \rightleftarrows \cS \colon x_*$) the stalk $\cO_x \coloneqq x\inv \cO$ has to take \'etale coverings of $k$-algebras of finite presentation to epimorphisms in $\rSet$. Unraveling the definitions, this means that for every \'etale cover $\{A \to A_i\}$ in $\cGetk$ and every solid diagram
	\[ \begin{tikzcd}
		{} & \coprod \Spec(A_i) \arrow{d} \\
		\Spec(\cO_x) \arrow{r} \arrow[dotted]{ur} & \Spec(A)
	\end{tikzcd} \]
	the lifting exists. This is a possible characterization of strictly henselian local rings (see \cite[Tag 04GG, condition (8)]{stacks-project}).
\end{eg}

As in the case of locally ringed spaces, we are not really interested in all the transformations of $\cG$-structures, but only in those that have a good local behavior. This can be made precise by introducing the notion of \emph{local transformation of $\cG$-structures}.
We recall that a morphism $f \colon \cO \to \cO'$ in $\Str_\cG(\cX)$ is said to be local if for every admissible morphism $f \colon U \to V$ in $\cG$ the induced square
\[ \begin{tikzcd}
	\cO(U) \arrow{r} \arrow{d} & \cO(V) \arrow{d} \\
	\cO'(U) \arrow{r} & \cO'(V)
\end{tikzcd} \]
is a pullback in $\cX$.
In the above example, the condition simply translates in the more familiar one of local morphism of local rings.

Precisely as in the case of locally ringed spaces, we can use $\cG$-structures and local morphisms of such to build an $\infty$-category of $\cG$-structured topoi, denoted $\Top(\cG)$.
The actual construction is rather involved, and we refer to \cite[Definition 1.4.8]{DAG-V} for the details.
Here, we shall content ourselves with the following rougher idea: the $\infty$-category $\Top(\cG)$ has as objects pairs $(\cX, \cO_\cX)$ where $\cX$ is an $\infty$-topos and $\cO_\cX$ is a $\cG$-structure on $\cX$, and as $1$-morphisms pairs $(f, \alpha) \colon (\cX, \cO_\cX) \to (\cY, \cO_\cY)$ where $f$ is a geometric morphism $f\inv \colon \cY \rightleftarrows \cX \colon f_*$ and $\alpha \colon f\inv \cO_\cY \to \cO_\cX$ is a \emph{local} transformation of $\cG$-structures on $\cX$.

The category $\Top(\cG)$ is too huge to be of any practical interest. Comparatively, it seems even huger than $\Sh(\dAff_k, \tauet)^\wedge$.
Therefore we are going to construct a full subcategory $\Sch(\cG)$ which morally corresponds to the subcategory of $\Sh(\dAff_k, \tauet)^\wedge$ spanned by geometric stacks. Stated in this way it is not quite a true statement, as we will see in discussing \cref{thm:comparison}, but until then it is a reasonable analogy.
The idea is not at all complicated: as schemes are a full subcategory of locally ringed spaces spanned by those objects which are locally isomorphic to special ones constructed out of commutative rings, so we will proceed in defining $\Sch(\cG)$.
As \cref{eg:discrete_etale_geometry} suggests, what we should try to do is construct a $\cG$-structured topos out of every object of $\mathrm{Ind}(\cG)$.
To keep the exposition at an elementary level, we will limit ourselves to consider the case of objects in $\cG$, and we refer the reader to \cite[§2.2]{DAG-V} for the general discussion.

Let $A \in \cG^{\mathrm{op}}$. We will denote by $A_{\mathrm{adm}}$ the small admissible site of $A$. The underlying $\infty$-category of $A_{\mathrm{adm}}$ is the opposite of the full subcategory of $\cG^{\mathrm{op}}_{A/}$ spanned by admissible morphisms $A \to B$. We then endow $A_{\mathrm{adm}}$ with the Grothendieck topology induced from the one on $\cG$, which we will still denote $\tau$.
Finally, we let $\cX_A$ be the \emph{non hypercomplete} $\infty$-topos of ($\cS$-valued) sheaves on $A_{\mathrm{adm}}$.
We next construct the $\cG$-structure on $\cX_A$.
There is a forgetful functor $A_{\mathrm{adm}} \to \cG$ which induces a composition
\[ A_{\mathrm{adm}}^{\mathrm{op}} \times \cG  \to \cG^{\mathrm{op}} \times \cG \xrightarrow{y} \cS \]
where $y$ is the functor classifying the Yoneda embedding, see \cite[§5.2.1]{Lurie_Higher_algebra}.
This corresponds to a functor
\[ \cO_A \colon \cG \to \PSh(A_{\mathrm{adm}}) \xrightarrow{\rL} \Sh(A_{\mathrm{adm}}, \tau) \]
where $\rL$ is the sheafification functor.
Note that if the Grothendieck topology on $\cG$ was subcanonical, there wouldn't be any need to apply $\rL$.
Observe further that $\cO_A$ is indeed left exact by the very construction.
We leave as an exercise to the reader to prove that $\cO_A$ takes $\tau$-coverings in effective epimorphisms (see \cite[Proposition 2.2.11]{DAG-V}).
Therefore the pair $(\cX_A, \cO_A)$ defines a $\cG$-structured topos, which we will denote as $\Spec^\cG(A)$.

\begin{rem}
	As it always happens in the $\infty$-categorical world, the construction of the functoriality is the most subtle point in the definition of an $\infty$-functor.
	It would rather hard if not impossible to explicitly exhibit $\Spec^\cG(-)$ as a functor $\cG \simeq (\cG^{\mathrm{op}})^{\mathrm{op}} \to \Top(\cG)$ if some alternative description wouldn't be available.
	We won't discuss the details, but, roughly speaking, the idea is to use the universal property of $\Spec^\cG(-)$ which describes it as a right adjoint to the global section functor $\Top(\cG) \to \mathrm{Ind}(\cG^{\mathrm{op}})$, informally defined by $(\cX, \cO_\cX) \mapsto \Map_\cX(\mathbf 1_\cX, \cO_\cX)$ (observe that the latter becomes a finite limit preserving functor $\cG \to \cS$ and therefore can be identified with an element of $\mathrm{Ind}(\cG^{\mathrm{op}})$). We refer the reader to \cite[§2.2]{DAG-V} (and especially to \cite[Theorem 2.2.12]{DAG-V}) for a detailed discussion.
\end{rem}

With these preparations, it is now easy to define $\Sch(\cG)$ as a full subcategory of $\Top(\cG)$.
We will say that a $\cG$-structured topos $(\cX, \cO_\cX)$ is a $\cG$-scheme (resp.\ a $\cG$-scheme locally of finite presentation) if there exists a collection of objects $U_i \in \cX$ satisfying the following two conditions:
\begin{enumerate}
	\item the joint morphism $\coprod U_i \to \mathbf 1_\cX$ is an effective epimorphism;
	\item for every index $i$, there exists an object $A_i \in \mathrm{Ind}(\cG^{\mathrm{op}})$ (resp.\ an object $A_i \in \cG^{\mathrm{op}}$) and an equivalence of $\cG$-structured topoi $(\cX_{/U_i}, \cO_\cX |_{U_i}) \simeq \Spec^\cG(A_i)$.
\end{enumerate}

We conclude this review with two important examples and some discussion about them.

\begin{eg} \label{eg:underived_DM_stacks}
	Let us go back to the geometry $\cGetk$ of \cref{eg:discrete_etale_geometry}. The category $\Sch(\cG)$ contains a very interesting full subcategory.
	To describe it, let us briefly recall that an $\infty$-topos $\cX$ is said to be $n$-localic (for $n \ge -1$ an integer) if it can be thought as the category of ($\cS$-valued) sheaves on some Grothendieck site $(\cC, \tau)$ with $\cG$ being an $n$-category (see our conventions on the meaning of this).
	We refer the reader \cite[§6.4.5]{HTT} for a more detailed account on this notion.
	Let $\Sch_{\le 1}(\cG)$ be the full subcategory of $\Sch(\cG)$ spanned by $\cG$-schemes $(\cX, \cO_\cX)$ such that $\cX$ is $1$-localic.
	Then \cite[Theorem 2.6.18]{DAG-V} shows that $\Sch_{\le 1}(\cG)$ is equivalent to the category of $1$-geometric (underived) \DM stacks.
	It will be a consequence of \cref{thm:comparison} that more generally $\Sch_{\le n}(\cG)$ is equivalent to the $\infty$-category of $n$-geometric $n$-truncated (underived) \DM stacks.
\end{eg}

\begin{eg} \label{eg:derived_DM_stacks}
	Let us define a new geometry $\cGetkder$ as follows. We let the underlying $\infty$-category of $\cGetkder$ to be the opposite of the full subcategory of $\mathrm{CAlg}_k$ spanned by compact objects.
	Observe that $\mathrm{CAlg}_k = \mathrm{Ind}(\cGetkder^{\mathrm{op}})$.	
	We will say that a morphism in $\cGetkder$ is admissible precisely when it is a (derived) \'etale morphism (see the previous section for the definition). We will further endow $\cGetkder$ with the (derived) \'etale topology, which we will still denote $\tauet$ (observe that if $A \to B$ is an \'etale map in the derived sense and the source is discrete, then so is the target).
	In this special case, we will write $\SpecEt$ instead of $\Spec^{\cGetkder}$.
	Following \cite[Definition 4.3.20]{DAG-V} (and using the important \cite[Proposition 4.3.15]{DAG-V}), we will say that a derived \DM stack (in the sense of \cite{DAG-V}) is a $\cGetkder$-scheme.
\end{eg}

The following theorem summarizes several results of \cite{DAG-V}. We report them here because it clarifies the relation between the above two examples:

\begin{thm}
	\begin{enumerate}
		\item \cite[Proposition 4.3.15]{DAG-V}The natural inclusion $\cTetk \to \cGetkder$ exhibits the latter as a geometric envelope of $\cTetk$;
		\item \cite[Remark 4.3.14 and Corollary 4.3.16]{DAG-V} the truncation functor $\pi_0 \colon \cGetkder \to \cGetk$ exhibits the latter as a $0$-stub for $\cGetkder$. In particular, the composition $\cTetk \to \cGetkder \to \cGetk$ exhibits $\cGetk$ as a $0$-truncated geometric envelope of $\cTetk$.
		\item \cite[Proposition 4.3.21]{DAG-V} the category of $1$-localic $\cGetk$-schemes is equivalent to the category of $\cGetkder$-schemes which are $1$-localic and $0$-truncated.
	\end{enumerate}
\end{thm}

\begin{rem}
	The derived \DM stacks of \cref{eg:derived_DM_stacks} are locally \emph{connective}.
	There is a non-connective variation of such objects, known as \emph{spectral \DM stacks}.
	This plays a major role in a certain branch of algebraic topology known as \emph{chromatic homotopy theory}.
	As we won't be concerned with such objects in this note, we invite the interested reader to consult \cite[§2, §8]{DAG-VII}.
	\cite[Corollary 9.28]{DAG-VII} completes the task of comparing the category of spectral \DM stacks with the one of \cref{eg:derived_DM_stacks}.
	We would like to draw the attention of the reader to the fact that characteristic $0$ is needed to have such a comparison.
	This is a complication that comes from the interaction with power operations in algebraic topology.
	In this note, no hypothesis on the characteristic is needed.
\end{rem}

\subsection{The main theorem}

We are finally ready to discuss the main comparison result.
In order to avoid confusion, we will refer from this moment on to derived \DM stacks as the geometric stacks for the HAG context $(\dAff_k, \tauet, \bP_\et)$ we discussed in \cref{subsec:HAG_II}, and to $\cGetkder$-schemes to the derived \DM stacks in the sense of \cite{DAG-V} we introduced in \cref{eg:derived_DM_stacks}.

Taking inspiration from the comparison discussed in \cref{eg:underived_DM_stacks}, we introduce the full subcategory $\Sch_{\le n}(\cGetkder)$ of $\Sch(\cGetkder)$ spanned by $\cGetkder$-schemes $(\cX, \cO_\cX)$ whose underlying $\infty$-topos $\cX$ is $n$-localic. We will further let $\Sch_{\mathrm{loc}}(\cGetkder)$ be the reunion of the $\infty$-categories $\Sch_{\le n}(\cGetkder)$ as $n$ varies.
The comparison result can therefore be stated as follows:

\begin{thm} \label{thm:comparison}
	There exists an equivalence of $\infty$-categories
	\[ \Phi \colon \Sch_{\mathrm{loc}}(\cGetkder) \rightleftarrows \mathbf{DM} \colon \Psi \]
	Moreover, for every $n \ge 1$, this restricts to an equivalence
	\[ \Sch_{\le n}(\cGetkder) \simeq \mathbf{DM}_n . \]
\end{thm}

The next section is entirely devoted to the proof of this theorem.

\begin{rem}
	The statement \cref{thm:comparison} is very similar to the one of \cite[Theorem 3.7]{Porta_GAGA_2015}.
	However, the proof of \cref{thm:comparison} is somehow subtler.
	One of the key points is that if $(\cX, \cO_\cX)$ is a derived \canal space (cf.\ \cite[Definition 12.3]{DAG-IX} or \cite[Definition 1.3]{Porta_GAGA_2015}), then the $\infty$-topos $\cX$ is always hypercomplete (see \cite[Lemma 3.2]{Porta_GAGA_2015}).
	This is false in the algebraic setting, and the reason is that if $A \in \mathrm{CAlg}_k$, then usually $\cX_A \coloneqq \Sh(A_\et)$ itself is not hypercomplete.
	As consequence, there is no direct analogue in this setting of \cite[Corollary 3.4]{Porta_GAGA_2015}: one needs to restrict himself to the case of localic $\cGetkder$-schemes to prove the corresponding statement (see \cref{prop:phi_X_hypercomplete}).
	
	Another important point that marks the difference is that if $A \in \mathrm{CAlg}_k$ then $\cX_A$ is $1$-localic instead of $0$-localic. Therefore the case of algebraic spaces has to be dealt with separately and it cannot be uniformly included in an induction proof. This is done in \cref{subsec:derived_algebraic_spaces}.
\end{rem}

\section{The proof of the comparison result}

We begin with the construction of the two functors $\Phi$ and $\Psi$.
\cite[Theorem 2.4.1]{DAG-V} provides us with a fully faithful embedding
\[ \phi \colon \Sch(\cGetkder) \to \Fun(\mathrm{Ind}(\cGetkder^{\mathrm{op}}), \cS) = \Fun(\dAff^{\mathrm{op}}, \cS) , \]
Unraveling the definition of $\phi$, we see that for $X = (\cX, \cO_\cX) \in \Sch(\cGetkder)$, the functor $\phi(X)$
\[ \phi(X) \colon \mathrm{CAlg}_k \to \cS \]
is defined informally by
\[ \phi(X)(A) = \Map_{\Sch(\cGetkder)}(\SpecEt(A), X) . \]
It follows from \cite[Lemma 2.4.13]{DAG-V} that this functor factors through $\Sh(\dAff_k, \tauet)$.

To obtain the functor $\Phi$ of \cref{thm:comparison}, we are left to show that the restriction of $\phi$ to $\Sch_{\mathrm{loc}}(\cGetkder)$ factors through $\mathbf{DM}$.
Let $X = (\cX, \cO_{\cX}) \in \mathrm{Sch}(\cGetkder)$. More specifically, the proof of \cref{thm:comparison} breaks into the following independent step:
\begin{enumerate}
	\item Let $n \ge 1$. If the underlying $\infty$-topos of $X$ is $n$-localic, then $\phi(X)$ is hypercomplete;
	\item Let $n \ge 1$. If the underlying $\infty$-topos of $X$ is $(n+1)$-localic, then $\phi(X)$ is $n$-geometric;
	\item The previous two points imply that $\phi$ factors through a fully faithful functor $\Phi \colon \Sch(\cGetkder) \to \mathbf{DM}$. Therefore, to achieve the proof, it will be sufficient to show that every object in $\mathbf{DM}$ arises is of the form $\phi(X)$ for $X \in \Sch_{\mathrm{loc}}(\cGetkder)$.
\end{enumerate}

We will deal with the first point in \cref{subsec:hypercompleteness}. In \cref{subsec:derived_algebraic_spaces} we will discuss the special case of derived algebraic spaces, which will serve as base for the proof by induction of the second point given in \cref{subsec:geomericity}.
Finally, we will treat the third point in \cref{subsec:essential_surjectivity}, thus achieving the proof of \cref{thm:comparison}.

\subsection{Hypercompleteness} \label{subsec:hypercompleteness}

Let us begin by a couple of preliminary lemmas.

\begin{lem} \label{lem:topos_theoretic_etale}
	Let $f \colon B \to A$ be a morphism in $\mathrm{CAlg}_k$ between finitely presented objects.
	The following conditions are equivalent:
	\begin{enumerate}
		\item $f$ is \'etale;
		\item the morphism $\Spec^{\mathrm{\acute{e}t}}(A) \to \Spec^{\mathrm{\acute{e}t}}(B)$ is \'etale in the sense of \cite[Definition 2.3.1]{DAG-V}.
	\end{enumerate}
\end{lem}

\begin{proof}
	A proof of this lemma can be formally deduced from \cite[Theorem 1.2.1]{DAG-VIII}.
	We will present here a shorter proof that works fine in the connective situation.
	The implication $1. \Rightarrow 2.$ is \cite[Example 2.3.8]{DAG-V}.
	Let us prove $2. \Rightarrow 1.$
	Since both $A$ and $B$ are finitely presented, we see that $\pi_0(A) \to \pi_0(B)$ is finitely presented.
	If we show that $\mathbb L_{A / B} \simeq 0$, we will obtain that $B \to A$ is finitely presented (in virtue of \cite[Proposition 8.8]{DAG-IX}\footnote{We warn the reader that there is a small mistake in \cite[Example 8.4]{DAG-IX}, when considering morphism of finite presentation to order $0$. Namely, it is not true that a discrete $A$-algebra $B$ is finitely generated if the canonical map $\colim \Hom_A(B, C_\alpha) \to \Hom_A(B,\colim C_\alpha)$ is injective for every filtered diagram $\{C_\alpha\}$ of $A$-algebras, the easiest counterexample being $A = \mathbb Z$ and $B = \mathbb Q$. However, the converse is true, and this is what is used afterwards. Therefore the subsequent results are not affected by this.}) and \'etale.
	
	Let
	\[ f\inv \colon \Sh(A_\et, \tauet) \to \Sh(B_\et, \tauet) \]
	be the inverse image functor.
	Consider the sheaf $\mathbb L_{\cO_A / f\inv \cO_B}$ on $A_{\mathrm{\acute{e}t}}$ defined by
	\[ C \mapsto \mathbb L_{\cO_A(C) / f\inv \cO_B(C)} = \mathbb L_{C / f\inv \cO_B(C)} \]
	Since the morphism $\Spec^{\mathrm{\acute{e}t}}(A) \to \Spec^{\mathrm{\acute{e}t}}(B)$ is \'etale, we see that $f\inv \cO_B \simeq \cO_A$.
	Therefore this sheaf is identically zero.
	
	On the other side, if $\eta\inv \colon \Sh(A_{\mathrm{\acute{e}t}}, \tauet) \to \cS$ is a geometric point, then
	\[ \eta\inv(\mathbb L_{\cO_A / f\inv \cO_B}) \simeq \mathbb L_{\eta\inv \cO_A / \eta\inv f\inv \cO_B} \]
	We can identify $\eta\inv f\inv \cO_B$ with a strictly henselian $B$-algebra $B'$.
	Since the map $B \to B'$ is formally \'etale, we conclude that
	\[ \mathbb L_{\eta\inv \cO_A / \eta\inv f\inv \cO_B} \simeq \mathbb L_{\eta\inv \cO_A / B} \]
	This is also the stalk of the sheaf on $A_{\mathrm{\acute{e}t}}$ defined by
	\[ C \mapsto \mathbb L_{C / B} \]
	Therefore, this sheaf vanishes as well. In particular, $\mathbb L_{A / B} \simeq 0$, completing the proof.
\end{proof}

\begin{lem} [{\cite[Lemma 1.3.5]{DAG-VIII}}] \label{lem:n_topoi}
	Let $\Top_{\le n}$ be the full subcategory of $\RTop$ spanned by $n$-localic $\infty$-topoi.
	Then $\Top_{\le n}$ is categorically equivalent to an $(n+1)$-category.
\end{lem}

\begin{proof}
	Let $n\textrm{-} \Top$ be the $\infty$-category spanned by $n$-topoi (see \cite[§6.4]{HTT}).
	Using the definition of $n$-localic $\infty$-topos we see that for $\cX, \cY \in \Top_{\le n}$, we have
	\[ \Map_{\Top_{\le n}}(\cX, \cY) \simeq \Map_{n \textrm{-} \Top}(\tau_{\le n-1} \cX, \tau_{\le n - 1} \cY) \to \Fun(\tau_{\le n - 1} \cX, \tau_{\le n - 1} \cY) \]
	Now, \cite[2.3.4.18]{HTT} shows that $\tau_{\le n - 1} \cY$ is (categorically equivalent to) an $n$-category, and therefore the simplicial set $\Fun(\tau_{\le n - 1} \cX, \tau_{\le n - 1} \cY)$ is (categorically equivalent to) an $n$-category as well in virtue of \cite[2.3.4.8]{HTT}. Invoking \cite[2.3.4.19]{HTT}, we conclude that the maximal Kan complex contained in $\Fun(\tau_{\le n - 1} \cX, \tau_{\le n - 1} \cY)$ is $n$-truncated.
	Since the map
	\[ \Map_{\Top_n}(\tau_{\le n-1} \cX, \tau_{\le n - 1} \cY) \to \Fun(\tau_{\le n - 1} \cX, \tau_{\le n - 1} \cY) \]
	is a monomorphism of simplicial sets, we see that the Kan complex
	\[ \Map_{\Top_n}(\tau_{\le n-1} \cX, \tau_{\le n - 1} \cY) \]
	is in fact an $n$-category.
	It follows again from \cite[2.3.4.19]{HTT} that it is $n$-truncated as well. In other words, $\Top_{\le n}$ is categorically equivalent to an $(n+1)$-category.
\end{proof}

\begin{prop} \label{prop:phi_X_hypercomplete}
	Let $X = (\cX, \cO_{\cX})$ be a $\cGetkder$-scheme and suppose that $\cX$ is $n$-localic, with $n \ge 1$.
	Then the functor $\phi(X) \colon \cC \to \cS$ is an hypercomplete sheaf.
\end{prop}

\begin{proof}
	Let $U^\bullet \to U$ be an \'etale hypercover in the category $\dAff_k$.
	Let $\Top_{\le n}(\cGetkder)$ be the $\infty$-category of $\cGetkder$-structured $\infty$-topoi which are $m$-localic for some $m \le n$.
	We claim that the geometric realization of the simplicial object $\SpecEt(U^\bullet)$ is $\Top_{\le n}(\cGetkder)$ is precisely $\SpecEt(U)$.
	The claim implies directly the lemma, since
	\begin{align*}
		\phi(X)(\SpecEt(U)) & = \Map_{\mathrm{Sch}(\cTetk)}(\SpecEt(U), X) \\
		& = \Map_{\Top_{\le n}(\cTetk)}(\SpecEt(U), X) \\
		& = \lim \Map_{\Top_{\le n}(\cTetk)}(\SpecEt(U^\bullet), X) \\
		& = \lim \phi(X)(\SpecEt(U^\bullet))
	\end{align*}

	We are therefore reduced to prove the claim. Let us denote by $\cX_U$ the topos of (non hypercomplete) sheaves on the small \'etale site of $U$.
	It follows from \cref{lem:topos_theoretic_etale} that each face map
	\[ \Spec^{\mathrm{\acute{e}t}}(U^n) \to \Spec^{\mathrm{\acute{e}t}}(U^{n-1}) \]
	is \'etale.
	Therefore, we can find objects $V^n \in \cX_U$ and identifications $\cX_{U^n} \simeq (\cX_U)_{/V^n}$.
	The universal property of \'etale subtopoi (see \cite[6.3.5.6]{HTT}), shows that we can arrange the $V^n$ into a simplicial object in $\cX_U$.
	At this point, we are reduced to show the following two statements:
	\begin{enumerate}
		\item in $\Top_{\le n}$ one has an equivalence $\cX_U \simeq \colim \cX_{U^\bullet}$;
		\item in $\cX_U$ one has an equivalence
		\[ \cO_U \simeq \varprojlim \cO_U |_{V^\bullet} \]
	\end{enumerate}
	Since $\Top_{\le n}$ is an $n$-category in virtue of \cref{lem:n_topoi}, \cref{prop:descent_vs_hyperdescent} shows that a presheaf with values in $\Top_{\le n}$ has descent if and only if it has hyperdescent. We are therefore reduced to the case where $U^\bullet$ is the \v{C}ech nerve of the map $U^0 \to U$.
	In this case, the general descent theory for $\infty$-topoi (see \cite[6.1.3.9]{HTT}) allows to conclude.
	As for the second statement, \cite[Theorem 4.3.32.(3)]{DAG-V} shows that the sheaf $\cO_U$ is hypercomplete as an object of $\cX_U$.
	The proof of the lemma is therefore achieved.
\end{proof}

\subsection{The case of algebraic spaces} \label{subsec:derived_algebraic_spaces}

Let $A \in \mathrm{CAlg}_k$.
We denote by $A_{\text{big, \'et}}$ the big \'etale site of $A$: that is, its underlying $\infty$-category is the opposite of $(\mathrm{CAlg}_k)_{A/}$, and the Grothendieck topology is the (derived) \'etale one.
There are continuous and cocontinuous morphisms of $\infty$-sites
\[ \begin{tikzcd}
(A_{\mathrm{\acute{e}t}}, \tauet) \arrow{r}{u} & (A_{\text{big, \'et}}, \tauet) \arrow{r}{v} & (\dAff_k, \tauet)
\end{tikzcd} \]
Observe that $u$ commutes with finite limits.
It follows (e.g.\ using \cite[6.1.5.2]{HTT}) that the induced adjunction
\[ u_s \colon \Sh(A_{\text{\'et}}, \tauet) \rightleftarrows \Sh(A_{\text{big, \'et}}, \tauet) \colon u^s \]
is a geometric morphism of $\infty$-topoi, in other words, $u_s$ commutes with finite limits.
Here $u^s$ denotes the restriction functor along $u$ and $u_s$ is obtained via the left Kan extension along $u$. We refer the reader to \cite[§2.4]{Porta_Yu_Higher_analytic_stacks_2014} for a more detailed discussion of the chosen notations and more specifically to \cite[Lemma 2.23]{Porta_Yu_Higher_analytic_stacks_2014} for the construction of the relevant adjunction.

In particular, we can use \cite[5.5.6.16]{HTT} to conclude that $u_s$ takes $n$-truncated objects to $n$-truncated objects.

This is not true for $v$, because it commutes only with connected limits.
However, we still have an adjunction
\[ v_s \colon \Sh(A_{\text{big, \'et}}, \tauet) \rightleftarrows \Sh(\dAff_k, \tauet) \colon v^s \]
which can be identified with the canonical adjunction
\[ v_s \colon \Sh(\dAff_k, \tauet)_{/\Spec(A)} \rightleftarrows \Sh(\dAff_k, \tauet) \colon v^s \]
where $\Spec(A)$ denotes the functor of points associated to $A$, accordingly to the notation introduced at the end of \cref{subsec:HAG_II}.

\begin{defin} \label{def:derived_etale_algebraic_space}
	Let $k$ be a commutative ring, $A$ a commutative $k$-algebra and $X \in \Sh(\mathrm{dAff}_k, \tauet)$ any sheaf equipped with a natural transformation $\alpha \colon X \to \Spec(A)$. We will say that \emph{$\alpha$ exhibits $X$ as an \'etale algebraic space over $\Spec(A)$} if there exists a $0$-truncated sheaf $F \in \Sh(A_\et, \tauet)$ and an equivalence $X \simeq v_s(u_s(F))$ in $\Sh(\dAff_k, \tauet)_{/\Spec(A)}$.
\end{defin}

\begin{rem} \label{rem:algebraic_spaces}
	The above definition is the analogue of \cite[Definition 2.6.4]{DAG-V} in the derived setting.
	Indeed, let us replace the $\infty$-category $\mathrm{CAlg}_k$ with the $1$-category $\mathrm{CRing}_k$.
	Keeping the same notations as above, we see that if $G \in \Sh(A_{\text{big, \'et}}, \tauet)$ then
	\[ v_s(G) = \coprod_{\phi \colon A \to B} G(\phi) \]
	If moreover $F$ is an object in $\Sh(A_\et, \tauet)$, then $(u_sF)(\phi) = \phi\inv(F)(B)$.
	In conclusion, we have
	\[ v_s(u_s(F))(B) = \{(\phi, \eta) \mid \phi \in \Hom_k(A, B), \eta \in (\phi\inv F)(B)\} \]
	This coincides precisely with the definition of $\widehat{F}$ given in \cite[Notation 2.6.2]{DAG-V}.
	A similar description holds true in the derived setting. Indeed, there is a natural transformation $v_s(u_s(F)) \to \Spec(A)$.
	The fiber over a given map $f \colon \Spec(B) \to \Spec(A)$ coincides precisely with the global sections of the discrete object $f\inv(F)$.
\end{rem}

The following proposition is the analogue of \cite[2.6.20]{DAG-V}. The proof is essentially unchanged:

\begin{prop} \label{prop:algebraic_spaces}
	Let $\alpha \colon Y \to \Spec(A)$ be a natural transformation of stacks.
	Write $\Spec^{\mathrm{\acute{e}t}}(A) = (\cX, \cO_{\cX})$.
	The following conditions are equivalent:
	\begin{enumerate}
		\item $\alpha$ exhibits $Y$ as a derived algebraic space over $\Spec(A)$;
		\item $Y$ is representable by a $\cGetkder$-scheme $(\cY, \cO_{\cY})$ and $\alpha$ induces an equivalence $(\cY, \cO_{\cY}) \simeq (\cX_{/U}, \cO_{\cX} |_U)$ for some discrete object $U \in \cX$.
		\item the morphism $\alpha$ is $0$-truncated and $0$-representable by \'etale maps.
	\end{enumerate}
\end{prop}

\begin{proof}
	We first prove the equivalence of (1) and (2).
	If $\alpha$ exhibits $Y$ as a derived algebraic space over $\Spec(A)$, we can find a $0$-truncated sheaf $U \in \Sh(A_{\mathrm{\acute{e}t}}, \tauet)$ and an equivalence $Y \simeq v_s(u_s(U))$ in $\Sh(\mathrm{dAff}, \tauet)_{/\Spec(A)}$.
	Now, \cite[Remark 2.3.4]{DAG-V} and \cref{rem:algebraic_spaces} show together that the functor represented by $(\cX_{/U}, \cO_{\cX} |_U)$ coincides with $Y$.
	Viceversa, if (2) is satisfied, then $U$ defines a derived algebraic space $v_s(u_s(U))$ over $\Spec(A)$, and \cite[Remark 2.3.4]{DAG-V} again allows to identify it with $Y$.
	
	Let us now prove the equivalence of (1) and (3) First, assume that (3) is satisfied.
	In this case, we can define a sheaf $U \colon A_{\mathrm{\acute{e}t}} \to \cS$ by sending an \'etale map $f \colon A \to B$ to the fiber product
	\[ \begin{tikzcd}
	U(B) \arrow{r} \arrow{d} & Y(B) \arrow{d}{\alpha_B} \\
	\{*\} \arrow{r}{f} & \Map(A,B)
	\end{tikzcd} \]
	Since $\alpha$ is $0$-truncated, we see that $U$ takes value in $\rSet$.
	Since it is obviously a sheaf, it defines a $0$-truncated object in $\Sh(A_{\mathrm{\acute{e}t}}, \tauet)$.
	\cite[Remark 2.3.4]{DAG-V} shows that $v_s(u_s(U))$ can be canonically identified with $Y$.
	
	Finally, let us prove that (1) implies (3).
	We already know that, in this situation, $\alpha$ is $0$-truncated.
	Choosing sections $\eta_\alpha \in Y(A_\alpha)$ which generate $Y$, we obtain an effective epimorphism
	\[ \coprod \Spec(A_\alpha) \to v_s(u_s(Y)) \]
	in $\Sh(\mathrm{dAff}_k, \tauet)$.
	Suppose that there exists a $(-1)$-truncated morphism $v_s(u_s(Y)) \to \Spec(B)$ for some $B \in \mathrm{CAlg}_k$.
	In this case, we see that
	\[ \Spec(A_\alpha) \times_{v_s(u_s(Y))} \Spec(A_\beta) \simeq \Spec(A_\alpha) \times_{\Spec(B)} \Spec(A_\beta) \simeq \Spec(A_\alpha \otimes_B A_\beta) \]
	In the general case, each fiber product $Y_{\alpha, \beta} \coloneqq \Spec(A_\alpha) \times_{v_s(u_s(Y))} \Spec(A_\beta)$ is again a derived algebraic space \'etale over $A$. We claim moreover that the canonical morphism $Y_{\alpha, \beta} \to \Spec(A_\alpha \otimes_A A_\beta)$ is $(-1)$-truncated.
	Assuming the claim, it follows that $Y_{\alpha, \beta} \to \Spec(A)$ is $(-1)$-representable by \'etale maps, hence it would follow that the morphism $\Spec(A_\alpha) \to v_s(u_s(Y))$ is $0$-representable.
	Finally, we see that it is representable by \'etale maps combining the equivalence between (1) and (2) with \cref{lem:topos_theoretic_etale}.
	
	We are left to prove the claim. Fix $f_\alpha \colon A_\alpha \to B$, $f_\beta \colon A_\beta \to B$ together with a homotopy making the diagram
	\[ \begin{tikzcd}
	A \arrow{r} \arrow{d} & A_\alpha \arrow{d}{f_\alpha} \\
	A_\beta \arrow{r}{f_\beta} & B
	\end{tikzcd} \]
	commutative.
	We have pullback squares
	\[ \begin{tikzcd}
	Y_{\alpha, \beta} \arrow{r} \arrow{d} & v_s(u_s(Y)) \arrow{d} \\
	\Spec(A_\alpha) \times \Spec(A_\beta) \arrow{r} & v_s(u_s(Y)) \times_{\Spec(A)} v_s(u_s(Y))
	\end{tikzcd} \]
	and since $\alpha \colon v_s(u_s(Y)) \to \Spec(A)$ is $0$-truncated, the statement follows.
\end{proof}

\subsection{$\phi(X)$ is geometric} \label{subsec:geomericity}

We can now prove that if $X \in \Sch_{\le n+1}(\cGetkder)$, then $\phi(X)$ is $n$-geometric.
The proof will go by induction, and \cref{prop:algebraic_spaces} will serve as basis of the induction.
Before doing that, however, it is convenient to prove the following lemma:

\begin{lem} \label{lem:decreasing_truncated_level}
	Let $n \ge 0$ be an integer. Fix $X = (\cX, \cO_\cX) \in \Sch_{\le n+1}(\cGetkder)$ and let $V \in \cX$ be an object such that $(\cX_{/V}, \cO_\cX |_V) \simeq \SpecEt(A)$ for some $A \in \mathrm{CAlg}_k$.
	Then $V$ is $n$-truncated.
\end{lem}

\begin{proof}
	We start by replacing $X$ with $\mathrm t_0(X) \coloneqq (\cX, \pi_0 \cO_X)$, which is a $\cGetk$-scheme in virtue of \cite[Corollary 4.3.30]{DAG-V}.
	We can therefore replace $A$ by $\pi_0(A)$ (observe also that $\SpecEt(\pi_0(A)) \simeq \Spec^{\cGetk}(\pi_0(A))$).
	
	Let us denote by $F_X \colon \mathrm{CRing}_k \to \cS$ the (truncated) functor of points associated to $F$.
	Similarly, let $F_V \colon \mathrm{CRing}_k \to \cS$ be the functor of points associated to $(\cX_{/V}, \cO_\cX |_V)$.
	The hypothesis shows that $F_V$ is nothing but the functor of points associated to $\pi_0(A)$ (with the notations of \cite{HAG-II}, this would be $\mathrm t_0( \Spec(\pi_0(A)) )$).
	Reasoning as in the proof of \cite[Theorem 2.6.18]{DAG-V}, we see that to prove that $V$ is $n$-truncated is equivalent to prove that for every (discrete) $k$-algebra $B$ the fibers of $F_V(B) \to F_X(B)$ are $n$-truncated.
	\cite[Lemma 2.6.19]{DAG-V} shows that $F(B)$ is $(n+1)$-truncated for every $k$-algebra $B$.
	On the other side, $F_V(B)$ is discrete by hypothesis.
	It follows from the long exact sequence of homotopy groups that the fibers of $F_V(B) \to F_X(B)$ are $n$-truncated, thus completing the proof.
\end{proof}

\begin{prop} \label{prop:phi_X_geometric}
	Let $X = (\cX, \cO_X) \in \Sch(\cGetkder)$ and suppose that $\cX$ is $n$-localic for $n \ge 1$.
	Then the stack $\phi(X)$ is $n$-geometric and moreover its truncation $\mathrm t_0 (\phi(X))$ is $n$-truncated.
\end{prop}

\begin{proof}
	The fact that $\mathrm t_0(\phi(X))$ is $n$-truncated follows directly from \cite[Lemma 2.6.19]{DAG-V}.
	
	Suppose now that $X = (\cX, \cO_\cX)$ is an $n$-localic $\cGetkder$-scheme.
	By definition, we can find a collection of objects $V_i \in \cX$ such that:
	\begin{enumerate}
		\item the morphism $\coprod V_i \to \mathbf 1_{\cX}$ is an effective epimorphism;
		\item the $\cGetkder$-schemes $(\cX_{/V_i}, \cO_X |_{V_i})$ are equivalent to $\Spec^{\mathrm{\acute{e}t}}(U_i)$ for $U_i \in \dAff_k$, and each $U_i$ is of finite presentation.
	\end{enumerate}
	Set $V \coloneqq \coprod V_i$.
	By functoriality, we obtain a map
	\[ \coprod \phi(V_i) \to \phi(X) \]
	We only need to show that this map is $(n-1)$-representable by \'etale morphisms and that it is an effective epimorphism.
	The second statement is an immediate consequence of \cite[Lemma 2.4.13]{DAG-V}.
	
	Suppose first that $X \simeq \SpecEt(A)$. In this case, the universal property of $\Spec^{\mathrm{\acute{e}t}}$ proved in \cite[§2.2]{DAG-V} shows that $\phi(X) = \Spec(A)$, and therefore $\phi(X)$ is $(-1)$-geometric.
	Now suppose that $X$ is a general $n$-localic $\cGetkder$-scheme.
	Since $\phi$ commutes with fiber products and is fully faithful, we see that for every map $\Spec(B) = \phi( \SpecEt(B) ) \to X$, one has
	\[ \Spec(B) \times_{\phi(X)} \phi(V_i) \simeq \phi(\SpecEt(B) \times_{(\cX, \cO_X)} (\cX_{/V_i}, \cO_X |_{V_i} )) \]
	Let $(f_*, \varphi) \colon \SpecEt(B) \to (\cX, \cO_X)$ be the given map.
	Then the fiber product $\SpecEt(B) \times_{(\cX, \cO_X)} (\cX_{/V_i}, \cO_X |_{V_i} )$ is the \'etale map to $\SpecEt(B)$ classified by the object $f\inv(V_i) \in \cX_A$, as it easily follows from \cite[6.3.5.8]{HTT}.

	We will complete the proof proving by induction on $n$ that each morphism $\phi(\cX_{/V_i}, \cO_\cX |_{V_i}) \to \phi(X)$ is $(n-1)$-representable by \'etale maps.
	If $n = 1$, \cref{lem:decreasing_truncated_level} shows that each object $V_i$ is $0$-truncated.
	It follows from \cref{prop:algebraic_spaces} that the fiber product $\Spec(A) \times_{\phi(X)} \phi(V_i)$ is $0$-geometric.
	Therefore, $\phi(X)$ is $1$-geometric.
	Now suppose that $\cX$ is $n$-localic for $n > 1$.
	\Cref{lem:decreasing_truncated_level} again shows that each $V_i$ is $(n-1)$-truncated, and therefore \cite[Lemma 2.3.16]{DAG-V} shows that the underlying $\infty$-topos of
	\[ \Spec^{\mathrm{\acute{e}t}}(A) \times_{(\cX, \cO_X)} (\cX_{/V_i}, \cO_X |_{V_i} ) \]
	is $(n-1)$-localic.
	The inductive hypothesis shows therefore that its image via the functor $\phi$ is $(n-1)$-geometric, and that the map to $\Spec(A)$ is \'etale.
	The proof is therefore complete.	
\end{proof}

\subsection{Essential surjectivity} \label{subsec:essential_surjectivity}

We finally prove that $\phi$ is essentially surjective.
Let $X \in \mathbf{DM}$ be $n$-geometric and suppose that $\mathrm t_0(X)$ is $n$-truncated.
It follows that the small \'etale site $(t_0(X))_{\mathrm{\acute{e}t}}$ is equivalent to an $n$-category.
Recall that there is an equivalence of $\infty$-categories
\[ X_\et \leftrightarrows (\mathrm t_0(X))_\et \]
(one can proceed as in \cite[Proposition 3.16]{Porta_GAGA_2015} using as base of the induction \cite[Corollary 2.2.2.10]{HAG-II}).
We conclude that $X_\et$ is an $n$-category.
In particular, the $\infty$-topos $\cX \coloneqq \Sh(X_\et, \tauet)$ is $n$-localic.
Define a $\cGetkder$-structure on $\cX$ as follows. Introduce the functor
\[ \cGetkder \times (X_\et)^{\mathrm{op}} \to \cS \]
defined as
\[ (U, V) \mapsto \Map_{\dAff_k}(V, U) \]
Fix $U \in \cGetkder$.
Since the Grothendieck topology on $\dAff_k$ is hyper-subcanonical, we see that the resulting object of $\Fun((X_{\mathrm{\acute{e}t}})^{\mathrm{op}}, \cS)$ is a hyper-sheaf.
In particular, we obtain a well defined functor
\[ \cO_X \colon \cTet \to \Sh(X_{\mathrm{\acute{e}t}}, \tauet) \]
that in fact factors through hypercompletion of this category.
In order to show that it is a $\cTet$-structure, we only need to check the following statements:
\begin{enumerate}
	\item $\cO_X$ is left exact;
	\item $\cO_X$ takes $\tauet$-coverings to effective epimorphisms.
\end{enumerate}
Since limits in $\Sh(X_{\mathrm{\acute{e}t}}, \tauet)$ are computed objectwise, the first two statements follow directly from the definition of $\cO_X$.
We are left to show that $\cO_X$ takes $\tauet$-coverings to effective epimorphisms.
Let $\{U_i \to U\}$ be a $\tauet$-cover in $\cTetk$.
We have to show that the morphism
\[ \coprod \cO_X(U_i) \to \cO_X(U) \]
is an effective epimorphism. In other words, we have to show that
\[ \coprod \pi_0 \cO_X(U_i) \to \pi_0 \cO_X(U) \]
is an epimorphism of sheaves of sets.
If $V \in X_\et$, an element in $(\pi_0 \cO_X(U))(V)$ is an \'etale covering $V_j \to V$ plus morphisms $V_j \to U$.
For each index $j$, we can find an \'etale covering $W_{jl} \to V_j$ such that the morphism $W_{jl} \to U$ factors through the cover $U_i \to U$.
Therefore, up to refining the cover $V_j \to V$, we see that the element in $(\pi_0 \cO_X(U))(V)$ comes from the coproduct.

We therefore conclude that $\cO_X$ is a hypercomplete $\cTetk$-structure on $\cX$.
Since $\cGetkder$ is a geometric envelope for $\cTetk$, we can identify $\cO_X$ with a $\cGetkder$-structure on $\cX$.

\begin{prop} \label{prop:Get_scheme}
	The pair $(\cX, \cO_X)$ is a $\cGetkder$-scheme.
\end{prop}

\begin{proof}
	Choose an \'etale atlas $p \colon \coprod U_i \to X$ in the category $\mathbf{DM}$.
	Since each morphism $p_i \colon U_i \to X$ is \'etale, we see each of them defines an element in the small \'etale site $(X_{\mathrm{\acute{e}t}}, \tauet)$.
	Since this site is subcanonical, we can identify each $U_i$ with objects $V_i \in \cX$.
	Moreover, the \'etale subtopos $(\cX_{/V_i}, \cO_X |_{V_i})$ is canonically identified with $(\Sh((U_i)_{\mathrm{\acute{e}t}}, \tauet), \cO_{U_i})$.
	The construction of the (absolute) spectrum functor of \cite[§2.2]{DAG-V}, shows that
	\[ \Spec^{\mathrm{\acute{e}t}}(U_i) \simeq (\Sh((U_i)_{\mathrm{\acute{e}t}}, \tauet), \cO_{U_i}) \]
	It will therefore be sufficient to show that the morphism $\coprod V_i \to \mathbf 1_{\cX}$ is an effective epimorphism.
	In order to do this it will be convenient to replace the small \'etale site $X_{\mathrm{\acute{e}t}}$ with the site $((\mathrm{Geom}^{\le n}_{/X})_{\mathrm{\acute{e}t}}, \tauet)$ of \'etale maps $Y \to X$ where $Y$ is an $n$-geometric $n$-truncated stack.
	We claim that the natural inclusion
	\[ (X_{\mathrm{\acute{e}t}}, \tauet) \to ((\mathrm{Geom}^{\le n}_{/X})_{\mathrm{\acute{e}t}}, \tauet) \]
	is an equivalence of sites in virtue of \cite[Lemma 2.34]{Porta_Yu_Higher_analytic_stacks_2014}.
	Indeed, even though the cited lemma actually works only in the hypercomplete setting, we can easily adapt it to the present situation as follows: the mapping spaces in $(\mathrm{Geom}^{\le n}_{/X})_\et$ are $n$-truncated, hence this is (categorically equivalent to) an $n$-category.
	Therefore the category of (non hypercomplete) sheaves on this site is an $n$-localic topos.
	The same goes for $\Sh(X_\et, \tauet)$, as we already discussed.
	Therefore, in order to check that the induced adjunction is an equivalence of $\infty$-categories, it is enough to check that the restriction to $n$-truncated object is an equivalence. This follows from the cited lemma, since we know that this morphism of sites induces an equivalence on all hypercomplete objects.
	
	In this way, we see that $\mathbf 1_{\cX}$ is the representable sheaf associated to the identity map $\mathrm{id}_X \colon X \to X$.
	We are therefore left to show that
	\[ \coprod \pi_0 \Map(-, U_i) \to \pi_0 \Map(-, X) \]
	is an epimorphism of sheaves on $((\mathrm{Geom}^{\le n}_{/X})_{\mathrm{\acute{e}t}}, \tauet)$.
	This follows immediately from the fact that the maps $U_i \to X$ were an atlas for $X$.
\end{proof}

We are left to prove that $\phi(\cX, \cO_X) \simeq X$.
We can proceed by induction on the geometric level $n$ of $X$.
If $n = -1$, the statement is obvious.
Otherwise, let $U_i \to X$ be an \'etale atlas for $X$.
Let $U \coloneqq \coprod U_i$ and let $U^\bullet$ be the \v{C}ech nerve of $U \to X$.
Combining the proof of \cref{prop:Get_scheme}, \cref{prop:phi_X_geometric} and the induction hypothesis, we see that $U^\bullet$ is a groupoid presentation for both $X$ and $\phi(\cX, \cO_X)$.
We therefore proved that the essential image of the functor
\[ \phi \colon \mathrm{Sch}(\cGetkder) \to \Sh(\dAff_k, \tauet) \]
contains all the \DM stacks in the sense of \cite{HAG-II}.

\appendix

\section{Descent vs hyperdescent}

The goal of this section is to prove the following folklore result:

\begin{prop} \label{prop:descent_vs_hyperdescent}
	Let $(\cC, \tau)$ be an $\infty$-Grothendieck site and let $\cD$ be an $(n+1,1)$-category.
	Then A functor $F \colon \cC^{\mathrm{op}} \to \cD$ satisfies descent if and only if it satisfies hyperdescent.
\end{prop}

\begin{proof}
	Let $D \in \cD$ be any object and let $c_D \colon \cD \to \cS$ be the functor \emph{corepresented} by $D$.
	Then $F$ satisfies descent (resp.\ hyperdescent) if and only if $c_D \circ F$ does.
	Since $\cD$ is an $(n+1,1)$-category, we see that $c_D \circ F$ takes values in $\tau_{\le n} \cS$. 
	Therefore, we may replace $\cD$ with $\cS$ and suppose that $F$ takes values in the full subcategory of $n$-truncated objects.
	For every $U \in \cC$, let us denote by $h_U$ the sheafification of the presheaf associated to $U$.
	Since $F$ is an $n$-truncated object, we see that
	\[ \Map_{\Sh_{\le n}(\cC, \tau)}(\tau_{\le n} h_U, F) \simeq \Map_{\Sh(\cC, \tau)}(h_U, F) \simeq F(U) \]
	where the last equivalence is obtained combining the universal property of the sheafification with the Yoneda lemma.
	Therefore, it will be sufficient to show that for every hypercover $U^\bullet \to U$ in $\cC$, the augmented simplicial diagram
	\[ \tau_{\le n} h_{U^\bullet} \to \tau_{\le n} h_U \]
	is a colimit diagram in $\Sh_{\le n}(\cC, \tau)$.
	Since $\tau_{\le n}$ is a left adjoint, we see that in $\Sh_{\le n}(\cC, \tau)$ the relation
	\[ |\tau_{\le n} h_{U^\bullet}| \simeq \tau_{\le n} |h_{U^\bullet}| \]
	holds.
	Moreover, since $U^\bullet \to U$ is an hypercover, the morphism $|h_{U^\bullet}| \to h_U$ is $\infty$-connected in virtue of \cite[6.5.3.11]{HTT}.
	Since $\tau_{\le n}$ commutes with $\infty$-connected morphisms, \personal{Indeed, if $f \colon F \to G$ is $\infty$-connected, then $\tau_{\le n}(f)$ is an equivalence on $\pi_i$ for $i \le n$ - because $\tau_{\le n}$ doesn't change such homotopy groups - and it is clearly an equivalence on $\pi_i$ if $i > n$ - for these groups became trivial.} we conclude that
	\[ \tau_{\le n} |h_{U^\bullet}| \to \tau_{\le n} h_U \]
	is an $\infty$-connected morphism between $n$-truncated objects.
	Therefore it is an equivalence in $\Sh(\cC, \tau)$.
	In conclusion, the morphism $|\tau_{\le n} h_{U^\bullet}| \to \tau_{\le n} h_U$ is an equivalence in $\Sh_{\le n}(\cC, \tau)$.
	The proof is now complete.
\end{proof}

\bibliographystyle{plain}
\bibliography{dahema}

\end{document}